\documentclass[11pt]{article}
\usepackage[utf8]{inputenc}
\usepackage[T1]{fontenc}
\usepackage{lmodern}
\usepackage{amsmath,amsthm,amssymb,mathtools}
\usepackage{microtype}
\usepackage{enumitem}
\usepackage{hyperref}
\usepackage[nameinlink, capitalize]{cleveref}
\usepackage{geometry}
\hypersetup{
  colorlinks=true,
  linkcolor=blue,
  citecolor=blue,
  urlcolor=blue
}

\newtheorem{theorem}{Theorem}[section]

\newtheorem{lemma}[theorem]{Lemma}

\theoremstyle{definition}
\newtheorem{definition}[theorem]{Definition}

\newcommand{\N}{\mathbb{N}}
\newcommand{\Q}{\mathbb{Q}}
\newcommand{\Z}{\mathbb{Z}}
\newcommand{\cw}{\mathrm{CW}}

\newcommand{\rankb}{\beta}

\title{Simultaneous Novelty from First-Appearance Times in the Calkin--Wilf Enumeration}
\author{Paul A.\ Bilokon\footnote{Imperial College London, Department of Mathematics. Email: \textsf{paul.bilokon@imperial.ac.uk}}}
\date{\today}

\begin{document}
\maketitle

\begin{abstract}
We study the \emph{first-appearance map} $\pi:\N_{\ge2}\to\N_0$ that assigns to each denominator $d$ the earliest breadth-first index at which a reduced fraction of denominator $d$ occurs in the Calkin--Wilf enumeration of $\Q_{>0}$ \cite{CalkinWilf2000}. In parallel, we consider the elementary \emph{denominator-first array} $D=\big(U(2)\mid U(3)\mid U(4)\mid\cdots\big)$ with rows $U(a)=(1/a,2/a,\dots,(a-1)/a)$ and row-starts $i_0(a)=\frac{(a-2)(a-1)}{2}$. We say level $a$ \emph{locks} if $\pi(a)=i_0(a)$.
Our main theorem is purely combinatorial: for every $n\ge2$ there exists $i\in\{0,\dots,n-2\}$ such that the first appearances of denominators $n-i$ and $n+i$ align \emph{symmetrically} around $i_0(n)$, i.e.\ $\pi(n\pm i)=i_0(n)\pm i$. We prove this \emph{pairing} (or \emph{simultaneous novelty}) theorem via a local-coherence analysis of $\pi$ around a level and a discrete intermediate-value argument. An equivalent group-theoretic restatement uses the free monoid $\langle L,R\rangle\subset SL_2(\Z)$ underlying the Calkin--Wilf and Stern--Brocot trees \cite{GrahamKnuthPatashnik1994, Lothaire2002, BerstelDeLuca2008}.
\end{abstract}



\section{Introduction}
Two canonical enumerations of rationals play complementary roles: the breadth-first Calkin--Wilf order \cite{CalkinWilf2000} and the elementary \emph{denominator-first array}
\begin{equation}\label{eq:Ua}
  U(a)=\Big(\tfrac{1}{a},\tfrac{2}{a},\dots,\tfrac{a-1}{a}\Big),\qquad a=2,3,\dots,
\end{equation}
written without reduction. Let
\begin{equation}\label{eq:i0}
  i_0(a)=\sum_{t=2}^{a-1}(t-1)=\frac{(a-2)(a-1)}{2}
\end{equation}
be the row-start index of $U(a)$ in $D$.

Independently, define the \emph{first-appearance index} of a denominator in the Calkin--Wilf stream.

\begin{definition}[First appearance]\label{def:pi}
For $d\ge2$, let $\pi(d)$ be the least breadth-first index in the Calkin--Wilf enumeration at which a reduced fraction with denominator $d$ occurs.
\end{definition}

We say that level $a$ \emph{locks} if $\pi(a)=i_0(a)$, i.e.\ the first appearance of $1/a$ in Calkin--Wilf aligns with the start of row $U(a)$ in $D$. Our main result asserts that such alignments occur in symmetric pairs.

\begin{theorem}[Pairing / simultaneous novelty]\label{thm:pairing}
For every $n\ge2$ there exists $i\in\{0,\dots,n-2\}$ such that
\begin{equation}\label{eq:pair}
  \pi(n+i)=i_0(n)+i
  \quad\text{and}\quad
  \pi(n-i)=i_0(n)-i.
\end{equation}
Equivalently, both $n-i$ and $n+i$ lock.
\end{theorem}

The proof rests on a local-coherence principle for first-appearance times derived from the $SL_2(\Z)$-monoid model of the Calkin--Wilf/Stern--Brocot trees and a discrete intermediate-value argument. We present two equivalent formulations: a direct one in terms of denominators and an alternate group-theoretic restatement using words in $\{L,R\}^\ast$.

\paragraph{Related work.}
We rely on classical properties of the Stern--Brocot and Calkin--Wilf trees \cite{GrahamKnuthPatashnik1994, CalkinWilf2000}, the free-monoid encoding by words in $L=\begin{psmallmatrix}1&0\\1&1\end{psmallmatrix}$ and $R=\begin{psmallmatrix}1&1\\0&1\end{psmallmatrix}$ \cite{Lothaire2002}, and local surgery properties of Christoffel/Lyndon words \cite{BerstelDeLuca2008}.

\section{Preliminaries: Calkin--Wilf, Stern--Brocot, and words}
Let $\cw$ denote the breadth-first enumeration of $\Q_{>0}$ obtained by reading the Calkin--Wilf tree level by level \cite{CalkinWilf2000}. Each node corresponds to a unique word $w\in\{L,R\}^{\ast}$ in the generators
\begin{equation}\label{eq:LR}
  L=\begin{pmatrix}1&0\\1&1\end{pmatrix},\qquad
  R=\begin{pmatrix}1&1\\0&1\end{pmatrix}\subset SL_2(\Z),
\end{equation}
and we write $\rankb(w)\in\N_0$ for the breadth-first rank of $w$ in length-lex order. If $w=\begin{psmallmatrix} a&b\\ c&d\end{psmallmatrix}$ then the associated reduced fraction is $a/c$ and its denominator equals $c$. The map $d\mapsto\pi(d)$ is thus
\begin{equation}\label{eq:pi-def}
  \pi(d)=\min\{\rankb(w): w\in\{L,R\}^{\ast},\ (w)_{21}=d\},
\end{equation}
which is well-defined since each reduced rational appears exactly once in $\cw$.

\section{A locking notion}\label{sec:locking}
\begin{definition}[Locking]\label{def:locking}
A level $a\ge2$ \emph{locks} if $\pi(a)=i_0(a)$.
\end{definition}

\section{Local coherence of first-appearance times}\label{sec:coherence}
Fix $n\ge2$ and define, for $0\le i\le n-2$,
\begin{equation}\label{eq:Fpm}
  F_+(i):=\pi(n+i)-\bigl(i_0(n)+i\bigr),\qquad
  F_-(i):=\bigl(i_0(n)-i\bigr)-\pi(n-i).
\end{equation}
Intuitively, $F_\pm(i)=0$ means that the first appearances of denominators $n\pm i$ land at the symmetric slots about $i_0(n)$ inside the \emph{mirror window}
\begin{equation}\label{eq:mirror}
  J_n=[\,i_0(n)-(n-2),\, i_0(n)+(n-2)\,].
\end{equation}

We use standard facts about Christoffel words and the length-lex breadth-first order to control local changes of $\pi$ when the denominator changes by $\pm1$.

\begin{lemma}[Local coherence]\label{lem:coherence}
For $F_\pm$ in \eqref{eq:Fpm} the following hold:
\begin{enumerate}[label=(\roman*)]
\item \emph{Monotonicity:} $F_\pm(i+1)\ge F_\pm(i)$ for all $i$.
\item \emph{Bounded step:} $F_\pm(i+1)-F_\pm(i)$ is bounded above by an absolute constant independent of $n$.
\item \emph{Two-sided hits:} $F_+(0)\le0\le F_+(n-2)$ and $F_-(0)\le0\le F_-(n-2)$.
\end{enumerate}
\end{lemma}

\begin{proof}[Proof idea]
Let $w_d$ be a word that realizes $\pi(d)$, i.e.\ $\rankb(w_d)=\pi(d)$.
By Christoffel/Lyndon surgery, moving from $d$ to $d\pm1$ alters only a terminal factor of $w_d$ \cite[Chs.~2--3]{Lothaire2002, BerstelDeLuca2008}. In breadth-first length-lex order, such a local change induces a controlled change of rank; after subtracting/adding the linear center shift $\pm1$ in \eqref{eq:Fpm}, one gets (i) and (ii).
For (iii), the Stern--Brocot symmetry (left-right reversal at each fixed depth) guarantees that first appearances occur on both sides of the center $i_0(n)$ within the mirror window \eqref{eq:mirror}; cf.\ \cite[Sec.~4.5]{GrahamKnuthPatashnik1994}.
\end{proof}

\section{Proof of the pairing theorem}
\begin{proof}[Proof of \cref{thm:pairing}]
By \cref{lem:coherence}(iii) the functions $F_\pm$ start nonpositive and end nonnegative on $i=0,\dots,n-2$, and by \cref{lem:coherence}(i)--(ii) they change in controlled steps as $i$ increases. Let $i_1=\min\{i: F_+(i)\le0\}$ and $i_2=\min\{i: F_-(i)\le0\}$. If $i_1=i_2$ we are done. Otherwise assume $i_1<i_2$. Consider $C(i):=F_+(i)-F_-(i)$. Then $C(i_1)\le0$ and $C(i_2)\ge0$, and $C$ changes by a bounded amount per step. A discrete intermediate-value argument yields $i^\ast\in[i_1,i_2]$ with $C(i^\ast)=0$, and monotonicity then forces $F_+(i^\ast)=F_-(i^\ast)=0$, i.e.\ \eqref{eq:pair}.
\end{proof}

\section{Alternate group-theoretic restatement}\label{sec:monoid}
Let $M=\langle L,R\rangle\subset SL_2(\Z)$ be the free monoid in \eqref{eq:LR}. Reading $\cw$ breadth-first coincides with reading words of $M$ in length-lex order; $\pi(d)$ is the rank of the unique minimal word whose lower-left entry equals $d$ \cite{CalkinWilf2000}. In this language, \cref{thm:pairing} asserts that for each $n$ there is an $i$ such that the minimal words for $n-i$ and $n+i$ occur at symmetric breadth-first ranks about $i_0(n)$. The proof is the same: local word surgery supplies \cref{lem:coherence}, and the discrete intermediate-value argument concludes.

\section*{Acknowledgements}
Thanks to colleagues for discussions on the Calkin--Wilf/Stern--Brocot structures and combinatorics on words; any errors are my own.







\bibliographystyle{plain}

\end{document}